\documentclass[letterpaper, 10 pt, journal, twoside]{IEEEtran}

\IEEEoverridecommandlockouts                              

\usepackage{preamble} 

\usepackage[all=normal,paragraphs=tight,floats=tight,tracking=tight]{savetrees}

\author{Luzia Knoedler$^{1*}$, Oswin So$^{2*}$, Ji Yin$^{3}$, Mitchell Black$^{4}$,\\Zachary Serlin$^{4}$, Panagiotis Tsiotras$^{3}$, Javier Alonso-Mora$^{1}$, and Chuchu Fan$^{2}$
\thanks{Manuscript received: March, 11, 2025; Revised June, 5, 2025; Accepted July, 14, 2025.}
\thanks{\textcopyright\, 2025 IEEE.  Personal use of this material is permitted.  Permission from IEEE must be obtained for all other uses, in any current or future media, including reprinting/republishing this material for advertising or promotional purposes, creating new collective works, for resale or redistribution to servers or lists, or reuse of any copyrighted component of this work in other works.}%
\thanks{This paper was recommended for publication by Editor Clement Gosselin upon evaluation of the Associate Editor and Reviewers’ comments.}%
\thanks{Luzia Knoedler and Javier Alonso-Mora are supported by the Office of Naval Research Global, grant N62909-25-1-2027, project SECURE.}
\thanks{$^*$Both authors contributed equally to this work.}
\thanks{$^{1}$Luzia Knoedler and Javier Alonso-Mora are with the Department of Cognitive Robotics, Delft University of Technology, Delft, The Netherlands. {\tt\footnotesize \{l.knoedler, j.alonsomora\}@tudelft.nl}}%
\thanks{$^{2}$Oswin So and Chuchu Fan are with the Department of Aeronautics and Astronautics, Massachusetts Institute of Technology, Cambridge, MA, USA. {\tt\footnotesize \{oswinso, chuchu\}@mit.edu}}%
\thanks{$^{3}$Ji Yin and Panagiotis Tsiotras are with the D. Guggenheim School
of Aerospace Engineering, Georgia Institute of Technology, Atlanta, GA, USA. {\tt\footnotesize \{jyin81, tsiotras\}@gatech.edu}}%
\thanks{$^{4}$Mitchell Black and Zachary Serlin are with the MIT Lincoln Laboratory,  Cambridge, MA, USA. {\tt\footnotesize \{mitchell.black, zachary.serlin\}@ll.mit.edu}}%
\thanks{Digital Object Identifier (DOI): see top of this page.}
}

\title{Safety on the Fly: Constructing Robust Safety Filters\\ via Policy Control Barrier Functions at Runtime}

\begin{document}

\markboth{IEEE Robotics and Automation Letters. Preprint Version. Accepted July, 2025}
{Knoedler \MakeLowercase{\textit{et al.}}: Safety on the Fly: Constructing Robust Safety Filters via Policy Control Barrier Functions at Runtime} 

\maketitle

\begin{abstract}
Control Barrier Functions (CBFs) have proven to be an effective tool for performing safe control synthesis for nonlinear systems.
However, guaranteeing safety in the presence of disturbances and input constraints for high relative degree systems is a difficult problem. 
In this work, we propose the Robust Policy CBF (RPCBF), a practical approach for constructing robust CBF approximations online via the estimation of a value function. 
We establish conditions under which the approximation qualifies as a valid CBF and demonstrate the effectiveness of the RPCBF-safety filter in simulation on a variety of high relative degree input-constrained systems.
Finally, we demonstrate the benefits of our method in compensating for model errors on a hardware quadcopter platform by treating the model errors as disturbances.
\noindent \hh{\textbf{Website including code: \href{https://www.oswinso.xyz/rpcbf/}{www.oswinso.xyz/rpcbf/}}}
\end{abstract}

\begin{IEEEkeywords}
Robot Safety, Optimization and Optimal Control, Collision Avoidance
\end{IEEEkeywords}

\section{Introduction and Related Works}
\IEEEPARstart{I}{n} the realm of autonomous systems, providing safety guarantees is crucial, especially in critical applications such as autonomous driving and healthcare robotics. 
\glspl*{cbf}~\cite{ames2016control, wieland2007constructive} have proven to be an effective tool to maintain and certify the safety of dynamical systems. In particular, they can be applied as a~\gls{sf} that minimally modifies arbitrary control inputs to ensure safety, making them especially valuable when integrated with learning-based controllers.

Despite their theoretical advantages, significant challenges remain in the practical application and construction of~\glspl{cbf}. First, constructing~\glspl{cbf} is non-trivial, specifically for high relative degree systems with input constraints. 
Second, the safety guarantees of ~\gls{cbf}-based controllers depend on having an accurate system model, which is rarely the case for systems in real life. 
This makes the safety guarantees of such controllers sensitive to model uncertainties.
\begin{figure}
    \centering
    \includegraphics[width=0.95\linewidth]{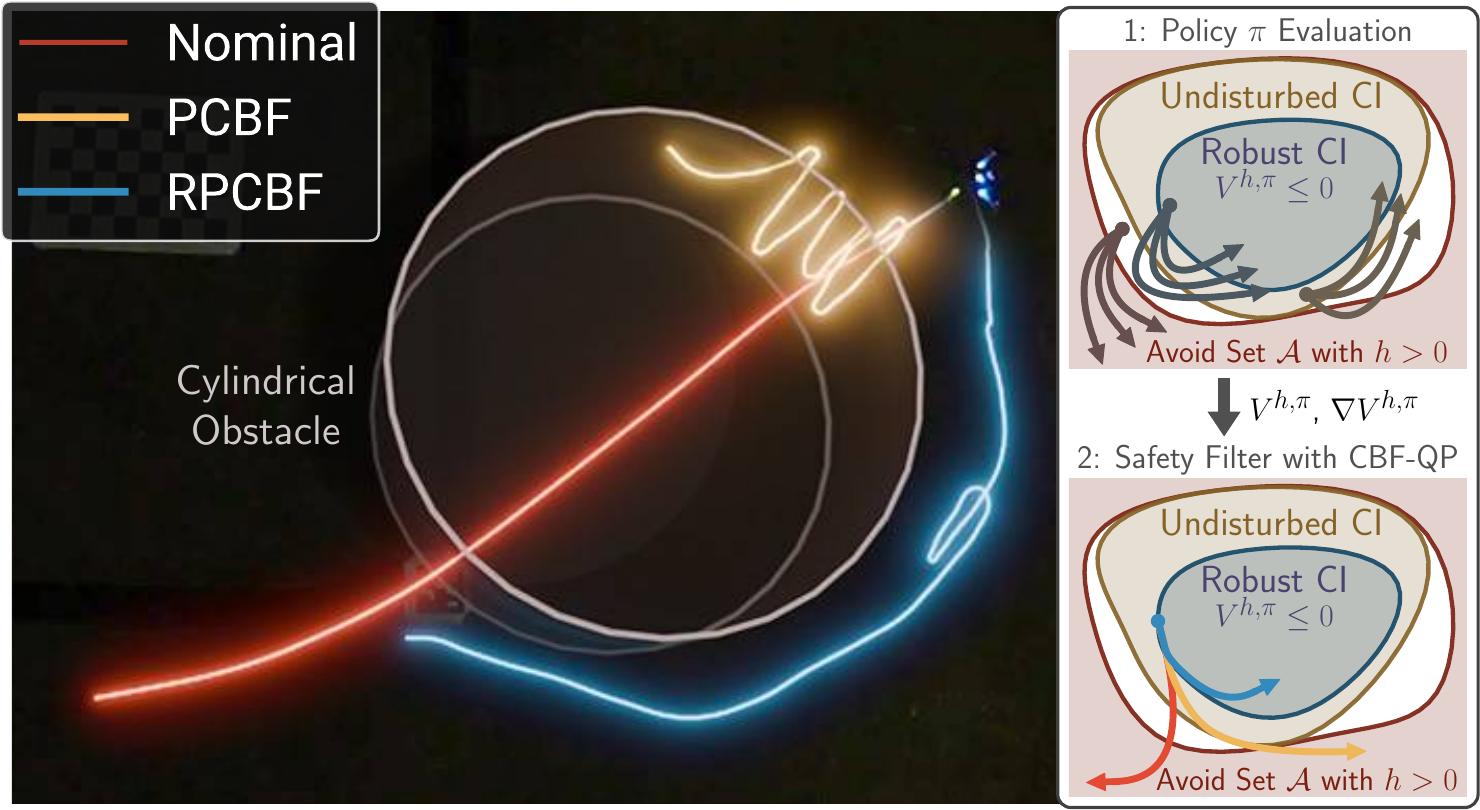}
    \caption{\small We propose the \glsfull{rpcbf}, which approximates the robust value function $\V^{h,\pi}$ of a system under bounded disturbances for policy $\pi$ \hh{at runtime}.
    The zero sublevel set of $\V^{h,\pi}$ is a robust controlled-invariant (CI) set. We apply a \textcolor{ggBlue}{\glsshort{rpcbf}}-\glsfull{sf} to ensure safety for any \textcolor{ggRed}{unsafe nominal policy}, demonstrating superior performance over the non-robust \textcolor{ggYellow}{PCBF}-\gls{sf} on a quadcopter with model errors treated as disturbances.
    }
    \label{fig:teaser}
    \vspace{-2em}
\end{figure}

\noindent \textbf{Learning Control Barrier Functions.}
To minimize reliance on extensive domain knowledge, a recent trend is to learn neural~\glspl{cbf} that approximate \glspl{cbf} using \glspl{nn}~\cite{dawson2023safe, lindemann2021learning, so2024train, saveriano2019learning, zhang2025gcbf+, srinivasan2020synthesis, wang2024simultaneous}. Neural \glspl{cbf} have been successfully applied to high-dimensional systems, including multi-agent control scenarios~\cite{zhang2025gcbf+, qin2021learning}, and have been extended to handle parametric uncertainties~\cite{dawson2022safe} and obstacles with unknown dynamics~\cite{yu2023sequential}. 
Although using~\glspl{nn} as~\glspl{cbf} offers universal approximation capabilities, it requires certifying them as valid~\glspl{cbf} to ensure safety guarantees and limits their interpretability.
Furthermore, using a naive approach to learning neural~\glspl{cbf} by minimizing a loss that encourages the~\gls{cbf} conditions can lead to a small or even empty forward-invariant set. Thus,~\cite{so2024train} presents a method to construct~\glspl{cbf} using policy evaluation of \emph{any} policy. They show that the policy value function is a~\gls{cbf} and learn an~\gls{nn} approximation. In this setting, the policy value function represents the maximum-over-time constraint violation, indicating how suitable a state is for a system following a specific policy. However, their approach does not consider uncertainties in the system dynamics.

\noindent \textbf{Robust Safety.} Controllers that are robust to disturbances are essential for ensuring the safety of autonomous systems in the real world.
This has been studied before in \textit{robust} \glspl{cbf} \cite{dawson2022safe,jankovic2018robust,cohen2022robust}, which guarantee safety under bounded disturbances. \rev{However, constructing robust \glspl{cbf} is inherently more difficult than constructing standard \glspl{cbf}, especially under input constraints.} Hamilton-Jacobi reachability analysis \cite{mitchell2005time} can be used to compute robust control-invariant sets, which can then be subsequently used for constructing robust \glspl{cbf}  \cite{choi2021robust, wabersich2023data,tonkens2022refining,tonkens2023patching}. However, reachability analysis in itself is challenging, with grid-based \hh{partial \gls{de}} solvers being limited to state dimensions below five \cite{mitchell2008flexible}, while deep learning-based solvers \cite{bansal2021deepreach,hsu2021safety,so2023solving} require subsequent \gls{nn} verification to check for solution accuracy. \rev{Moreover, both learning-based \gls{cbf} approaches and deep learning-based reachability solvers depend on predefined system dynamics and disturbance assumptions, which are difficult to adapt, limiting their flexibility once deployed, as retraining cannot be performed on the system. Thus, \cite{lin2024one} train a value network with the avoidance set and disturbance bounds as inputs, however this increases training data requirements and complicates evaluating how well the learned network represents the true value function.}

\hh{As an alternative to robust safety, other works focus on risk-aware safety, which aims to ensure safety with high probability by modeling disturbances probabilistically and incorporating risk measures \cite{ahmadi2021risk,liu2025risk}. Unlike robust methods, which ensure constraint satisfaction for all disturbances within a bounded set but may be overly conservative, risk-aware approaches typically rely on knowledge of the disturbance’s probability distribution. In this work, we focus on robust safety.}

\react{We propose a practical approach for constructing a \gls{cbf} approximation \hh{at runtime}, which can be derived for any system dynamics and disturbance bounds without requiring (re)training. We establish conditions under which the resulting \gls{cbf} approximation qualifies as a valid \gls{cbf}. Our method constructs \glspl{cbf} by evaluating the value function of \emph{any} policy, which has been shown to be a valid \gls{cbf} in~\cite{so2024train}. By leveraging finite-horizon policy rollouts, we enable a more detailed analysis of safety guarantees than \gls{nn} approximations. We apply this approach to construct approximations of robust \glspl{cbf}.}

\noindent \textbf{Contributions.} We summarize our contributions as follows.
\begin{enumerate}
    \item We propose a method of constructing \react{(robust)} \glspl{cbf} using the \react{(robust)} policy value function and a real-time approximation that can be used \hh{at runtime}. 
    \item We \rev{demonstrate real-time performance} and the benefits of our robust \glspl{cbf} on a hardware quadcopter, where robustness to model errors is key for collision prevention.
\end{enumerate}

\section{Preliminaries}\label{sec:preliminaries}
\subsection{Problem Statement}
We consider a \textit{disturbed} continuous-time, control-affine dynamical system of the form
\begin{equation}
    \dot{\vx}_t = f(\vx_t, \vd_t) + g(\vx_t, \vd_t)\vu_t,\label{eq:system}
\end{equation}
with state $\vx_t \in \cX \subseteq \bR^n$, control input $\vu_t \in \cU \subseteq \bR^m$ and unknown, bounded, smooth disturbance $\vd_\s{min} \leq \vd_t \leq \vd_\s{max}$ with $\vd_\s{min}, \vd_\s{max} \in \bR^d$ (\rev{e.g., estimated from empirical data), where $\vd_t$ can be time-varying.} 
The functions $f$ and $g$ are assumed to be locally Lipschitz continuous. 
Let $\cA \subset \cX$  denote the set of states to be avoided. This paper addresses the following~\gls{sf} synthesis problem:

\begin{problem}[Safety Filter Synthesis]\label{prob:1}
Given a system~\eqref{eq:system} and an avoid set $\cA \subset \cX$, find a control policy $\pi_\s{filt}: \cX \rightarrow \cU$ that keeps the system state outside $\cA$ while staying close to a performant, possibly unsafe nominal policy $\pi_\s{nom}: \cX \rightarrow \cU$:
\begin{equation}
 \begin{aligned}
    \min_{\pi_\s{filt}} \, &\lVert\pi_\s{filt} - \pi_\s{nom}\rVert\\
    \text{s.t.} \ &\hh{\dot{\vx}_t = f(\vx_t, \vd_t) + g(\vx_t, \vd_t)\pi_\s{filt}(\vx_t)}\\
    &\vx_t \notin \cA, \ \forall t\geq 0,\\
\end{aligned}   
\end{equation}
where $\lVert\cdot\rVert$ is some distance metric.
\end{problem}
\noindent We focus on solving \cref{prob:1} using (zeroing) \glspl*{cbf}~\cite{xu2015robustness}.

\subsection{Safety Filters using Control Barrier Functions}
We begin by providing a standard definition of a \gls*{cbf} in the non-robust case, which we extend to the robust case for~\glspl*{pcbf} in the next section.
Define the \textit{undisturbed} system to be a particular case of the disturbed system~\eqref{eq:system} without disturbances ($\vd=0$), by
\begin{equation}
    \dot{\vx}_t = f(\vx_t, 0) + g(\vx_t, 0) \vu_t.\label{eq:undisturbed_system}
\end{equation}
Let $\cbf: \cX \rightarrow \bR$ be a continuously differentiable function, with $\cC = \{\vx \in \cX \, | \, \cbf(\vx) \leq 0\}$ as its $0$-sublevel set. Let $\alpha: \bR \rightarrow \bR$ be an extended class-$\kappa_\infty$ function\footnote{Extended class-$\kappa_\infty$ is the set of continuous, strictly increasing functions $\alpha : (-\infty,\infty) \rightarrow (-\infty, \infty)$ with $\alpha(0) = 0$.}. Then, $\cbf$ is a \gls*{cbf} for the undisturbed system~\eqref{eq:undisturbed_system} on $\cX$~\cite{ames2016control} if
\begin{subequations}
\label{eq:cbf}
    \begin{flalign}
        \cbf(\vx) > 0, \ \forall \vx \in \cA, &\label{eq:cbf1}\\
        \cbf(\vx) \leq 0 \Rightarrow \inf_{\vu \in \cU}\Lie_f \cbf(\vx) + \Lie_g \cbf(\vx)\vu\leq -\alpha(\cbf(\vx)), &\label{eq:cbf2}
    \end{flalign}
\end{subequations}
with $\Lie_f \cbf \coloneq \nabla \cbf\T f$ and $\Lie_g \cbf \coloneq \nabla \cbf\T g$. It then follows that any control input $\vu \in K_\s{cbf}$ with 
$$
K_\s{cbf}(\vx) = \{\vu \in \cU \,|\, L_f\cbf(\vx) + L_g\cbf(\vx)\vu + \alpha(\cbf(\vx))\leq 0\}
$$
renders $\cC$ forward-invariant~\cite{ames2016control}.
In other words, there exists an $\vu \in \cU$ such that any trajectory starting within $\cC$ remains in $\cC$. 
Asymptotic stability of $\cC$ can be achieved by extending~\eqref{eq:cbf2} to hold for all $\vx \in \cX$~\cite{xu2015robustness}.
Since the right hand side of~\eqref{eq:cbf2} is linear in $\vu$, given a \gls{cbf} $B$,
we can solve \cref{prob:1} for \eqref{eq:undisturbed_system} using the following \gls*{qp}-based controller:
\begin{align*}
    \vu_\s{CBF-QP} = & \arg\min_\s{\vu \in \cU}  \ \lVert\vu - \pi_\s{nom}(\vx)\rVert^2 \tag{CBF-QP}\label{eq:cbf-qp} \\ 
    & \text{s.t.} \ \Lie_f\cbf(\vx) + \Lie_g \cbf(\vx)\vu \leq -\alpha(\cbf(\vx)).
\end{align*}
While \glspl*{cbf} can be applied to guarantee safety for a known undisturbed system, \rev{three} major challenges remain:
\begin{enumerate}
    \item How do we synthesize a valid \gls*{cbf} that satisfies \eqref{eq:cbf2} \rev{for high relative degree systems with input constraints?}
    \item \rev{How do we synthesize a robust \gls*{cbf} that ensures safe control for the disturbed system?}
    \item \rev{How can we efficiently derive a \gls*{cbf} \hh{at runtime} for different system dynamics and disturbance assumptions?}
\end{enumerate}

\section{\react{(Robust)} Policy Control Barrier Functions}\label{sec:method}
To address the above challenges, we leverage the insight from~\cite{so2024train} that~\glspl{cbf} can be constructed by deriving the policy value function through the evaluation of \emph{any} policy. Rather than approximating the policy value function with an~\gls{nn} as in \cite{so2024train}, we propose a \rev{real-time approximation} that avoids \glspl{nn} \rev{and can be derived \hh{at runtime}} through a finite-horizon numerical approximation.
We further extend this approach to the robust case and introduce \glspl{rpcbf} \rev{and subsequently propose a sampling-based approximation that can be derived \hh{at runtime}.}
Next, we revisit the formulation of \glspl*{pcbf} and describe our extensions and approximations.

\subsection{Constructing \acrshortpl{cbf} via Policy Evaluation}
Based on~\cite{so2024train}, we first derive the \gls{pcbf} formulation for the undisturbed system in~\eqref{eq:undisturbed_system}. 
Assume that the avoid set~$\cA$ can be described as the super-level set of a function $h: \cX \rightarrow \bR$
(e.g., the negative distance to the constraint):
\begin{equation}\label{eq:h}
    \cA = \{\vx \in \cX \,|\, h(\vx) >0\}.
\end{equation}
\hh{Note that $h(\vx) > 0$  for states that are already in the failure set, whereas $B(\vx) > 0$ for states from which failure is inevitable in the future under the given dynamics and input constraints. In the absence of input constraints, $h$ and $B$ may coincide.}
We denote by $\vx_t^\pi$ the resulting state at time~$t$ when starting from the initial state $\vx_0$ and following policy~$\pi : \cX \rightarrow \cU$. 
Furthermore, we define the \emph{maximum-over-time} value function for the undisturbed system in~\eqref{eq:undisturbed_system} as
\begin{equation}\label{eq:value_fn}
    \V^{h,\pi}_\infty(\vx_0)  \coloneqq \sup_{t \geq 0} h(\vx_t^\pi).
\end{equation}
As stated in~\cite[Theorem 1]{so2024train}, the \emph{policy value function}~$\V^{h,\pi}_\infty$ is a~\gls*{cbf} for the undisturbed system in~\eqref{eq:undisturbed_system} for any~$\pi$, since $\V^{h,\pi}_\infty$ satisfies the following two inequalities $\forall \vx \in \cX$
\begin{align}\label{eq:v}
    \V^{h,\pi}_\infty(\vx) &\geq h(\vx),\\
    \nabla\V^{h,\pi}_\infty(\vx)^T(f(\vx)+g(\vx)\pi(\vx))&\leq 0,
\end{align}
which imply \eqref{eq:cbf1} and \eqref{eq:cbf2}. For details, we refer to \cite{so2024train}.
The key intuition here is that~$\V^{h,\pi}_\infty$ provides an upper bound on the worst future constraint violation $h$ under the optimal policy since the optimal policy will do no worse than~$\pi$. Thus,~\glspl{cbf} can be constructed via policy evaluation of any policy. \rev{We refer to $\pi$ as the \dpolicy, noting that the nominal policy~$\pi_\s{nom}$ differs from the \dpolicy.}

\subsection{Finite Horizon Approximation of \glspl{pcbf}}\label{sec:finite_horizon}
A key challenge with the policy value function $V^{h,\pi}_\infty$ is that its definition requires an \textit{infinite-horizon}.
While \cite{so2024train} tackles this problem by using an~\gls{nn} to learn $V^{h,\pi}_\infty$ with a loss derived using dynamic programming, we take a different approach and perform a \textit{finite-horizon} approximation that can be computed \textit{without} the use of an~\gls{nn}, enabling a more in-depth analysis of the resulting safety guarantees. Expanding $V^{h,\pi}_\infty$: 
\begin{align}
       V^{h,\pi}_\infty(\vx_0)
       &= \max\Big\{ \sup_{0\leq t < T} h(\vx_t^\pi), V_\infty^{h,\pi}(\vx_T^\pi) \Big\} \label{eq:tmp:e} \\ 
      &\approx \sup_{0\leq t < T} h(\vx_t^\pi) \rev{\coloneqq V^{h,\pi}_T(\vx_0)}, \label{eq:finite_horizon_def}
\end{align}
where the approximation is made by dropping the $V_\infty^{h,\pi}(\vx_T)$ ``tail''.
The question is then whether the finite-horizon approximation $V_T^{h,\pi}$ is a~\gls{cbf} and can provide safety guarantees.

We can at least answer this in the affirmative when the approximation in \eqref{eq:finite_horizon_def} is an equality, i.e., the maximum occurs in~$[0, T\rev{)}$. 
We state this formally in the following theorem.
\begin{theorem}\label{th:1}
    Suppose that for all $\vx_0 \in \cX$,
    \begin{equation} \label{eq:tmp:d}
        V^{h,\pi}_T(\vx_0) \leq 0 \implies \sup_{0 \leq t < T} h(\vx_t^\pi) > V_\infty^{h,\pi}(\vx_T^\pi).
    \end{equation}
    Then, $V^{h,\pi}_T$ is a~\gls{cbf}.
\end{theorem}
\begin{proof}
    Since $V^{h,\pi}_T(\vx) \geq h(\vx)$ by definition, \eqref{eq:cbf1} is satisfied by $V^{h,\pi}_T$. Moreover, by \eqref{eq:tmp:d}, $V^{h,\pi}_T(\vx_0) = V^{h,\pi}_\infty(\vx_0)$ when $V^{h,\pi}_T(\vx_0) \leq 0$. Hence, since $V^{h,\pi}_\infty$ is a~\gls{cbf}, \eqref{eq:cbf2} holds for $V^{h,\pi}_{\infty}$, and thus also holds for $V^{h,\pi}_{T}$. Thus, $V^{h,\pi}_{T}$ is a~\gls{cbf}.
\end{proof}
\noindent This enables us to prove the following corollary.
\begin{corollary}\label{th:2}
    Suppose there exists a $\tilde{T} < \inf$ such that 
    \begin{equation} \label{eq:tmp:f}
        \argmax_{t \geq 0} h(\vx_t^\pi) < \tilde{T}, \quad \forall \vx_0 \text{ where } V_T^{h,\pi}(\vx_0) \leq 0.
    \end{equation}
    Then, $V^{h,\pi}_T$ is a CBF for $T \geq \tilde{T}$.
\end{corollary}
\begin{proof}
    \eqref{eq:tmp:f} implies \eqref{eq:tmp:d} for $T \geq \tilde{T}$. Proof  from \cref{th:1}.
\end{proof}

\rev{
The value of $\tilde{T}$ depends on the chosen \dpolicy~$\pi$. Any policy can be selected, but a conservative \gls{ci} set may result. If $\pi$ is chosen as a controller that steers the system towards a safe, steady-state or \gls{ci} set within a finite horizon $\tilde{T}$, \Cref{th:2} holds, thus $V_T^{h,\pi}$ is a valid \gls{cbf}. 
}

\hh{If $V^{h,\pi}_T(\vx_0)=0$, then applying the design policy $\pi$ exactly guarantees that the system will remain safe for at least the time horizon $0\leq t < T$, even if $V^{h,\pi}_T(\vx_0) < V^{h,\pi}_\infty(\vx_0)$.}

\begin{remark}[Connections to Backup Controller / \glspl{cbf}]
    Since the zero sublevel set of $V^{h,\pi}_\infty$ is a \gls{ci} set under $\pi$, \eqref{eq:tmp:e} can also be seen as Backup~\gls{cbf} \cite{singletary2022onboard,chen2021backup} with backup controller $\pi$ \hh{and no known \gls{ci} terminal set.} Unlike this (and other similar approaches \cite{agrawal2024gatekeeper}), our approach replaces the need for a known \gls{ci} set with the requirement of a sufficiently long horizon $T$. \hh{Thus, the design policy $\pi$ can be chosen arbitrarily and is not required to steer the system into a \gls{ci} set. Furthermore, we demonstrate that the naive approximation of $h$ over a time-discretized state trajectory introduces gradient errors. To address this, we present an improved time-discretization using cubic splines in \cref{sec:discrete}.}
\end{remark}
\begin{remark}[Connections to~\gls{mpc}]
    The finite-horizon approximation here is closely related to the use of~\gls{mpc} by practitioners.
    More precisely, although a terminal constraint set is often required to theoretically guarantee recursive feasibility of finite-horizon~\gls{mpc}~\cite{mayne2000constrained,brito2019model}, practitioners often apply~\gls{mpc} without the use of such a terminal constraint set to wide success \cite{kim2019highly,alcala2020autonomous,wang2021variational,so2022maximum}.
    Our decision to drop the $V^{h,\pi}_\infty(\vx_T^\pi)$ term can be viewed as being similar to dropping the terminal constraint set.
    Another similarity is the choice of horizon $T$. Namely, recursive feasibility holds in MPC given a sufficiently large horizon~\cite{boccia2014stability}, similar to \Cref{th:2}.
    \hh{However, the \gls{mpc} horizon length is limited, as it requires solving a potentially nonlinear and non-convex optimization problem online, with computational complexity typically scaling cubically with $T$~\cite{kirches2012efficient}.} \hhh{A key advantage of \gls{pcbf}-\glspl{sf} is that they only solve the simpler~\eqref{eq:cbf-qp}, whose computation time is unaffected by~$T$, see \Cref{sec:compute}.}
\end{remark}
\begin{figure}
    \centering
    \includegraphics[width=0.4\textwidth]{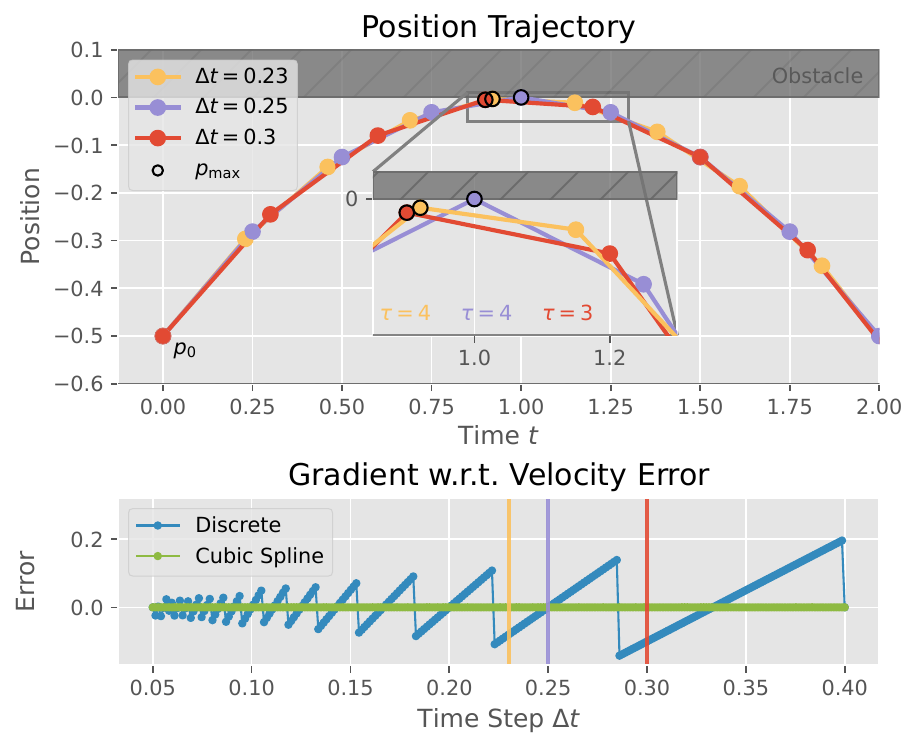}
    \caption{\small\rev{\textbf{Value Function Gradient Error for Discrete-Time Double Integrator.}}
    We highlight the discretized trajectory (top) and corresponding gradient error (bottom) for three different choices of $\Delta t$ (\textcolor{ggYellow}{yellow}, \textcolor{ggPurple}{purple}, \textcolor{ggRed}{red}).
    The gradient of the naive discrete-time approximation has large errors and varies with the choice of $\Delta t$. 
    Taking the maximum of the cubic spline leads to much smaller errors.
    }
    \label{fig:double_integrator_gradient}
    \vskip -5mm
\end{figure}
\begin{remark}[Connections to \gls{psf}~\cite{wabersich2021predictive}]
    \rev{The finite-horizon approximation is closely related to the \gls{psf}, which implicitly represents the safe set via a finite-horizon \gls{mpc} problem with terminal constraints or long horizons for recursive feasibility. 
    Unlike \gls{pcbf}-\glspl{sf}, the \gls{psf} requires solving a potentially nonlinear and nonconvex optimization problem online with complexity scaling cubically in $T$ \cite{kirches2012efficient}.
    \gls{pcbf}-\glspl{sf} only require solving the simpler~\eqref{eq:cbf-qp}.
    However, while the \gls{psf} may find a locally optimal solution, the conservativeness of the \gls{pcbf}-\gls{sf} depends on~$\pi$.}
\end{remark}

\subsection{Time Discretization of Policy Control Barrier Functions}\label{sec:discrete}
Another challenge lies in how to compute the maximum in \eqref{eq:finite_horizon_def}.
The states $\vx_t^\pi$ can be solved numerically using an \hh{ordinary \gls{de}} solver, resulting in a time-discretized state trajectory. 
It is tempting to then consider taking the maximum $h$ over this trajectory, i.e., for time discretization $\Delta t$,
\begin{equation} \label{eq:tmp:b}
    \V^{h,\pi}_T(\vx_0) \approx \max_{0\leq k < \horizon} h(\vx_{k \Delta t}^\pi).
\end{equation}
However, the gap between \eqref{eq:finite_horizon_def} and \eqref{eq:tmp:b} is particularly disastrous when computing the gradient.
We illustrate this in the following example for the~\gls{di}.

\noindent\textbf{Example: Gradient Error on the Double Integrator.}
Consider a~\gls{di} with positive velocity $v_0 > 0$ decelerating with $\pi(\vx) = a = -1$. The dynamics are defined by $\dot{p} = v$, $\dot{v} = a$ with initial state $\vx_0 = [p_0, v_0]$  and constraints $h(\vx) = p \leq 0$.
For the continuous-time case, the gradient can be derived as
\begin{equation}
    \nabla\V^{h,\pi}_\infty(\vx_0) = \frac{\partial p_\s{max}}{\partial \vx_0} = [1, v_0].
\end{equation}
After (exact) time discretization with timestep $\Delta t$, the time-discretized states can be computed as
\begin{subequations}
    \begin{flalign}
    p_k = p_0 +v_0 k\Delta t + 0.5a(k\Delta t) ^2, \\
    v_k = v_0 + ak\Delta t.
    \end{flalign}
\end{subequations}
We now show that the gradient of $V^{h,\pi}_\infty$ depends on $\Delta t$ and denote by $\nabla \V^{h,\pi}_{\infty, \Delta t}$ the resulting gradient. Let $\tau$ be the \textit{integer} time step $k$ at which the maximum position is reached. The maximum position is then given by
\begin{equation}
    V^{h,\pi}_{\infty, \Delta t} = p_\s{max} = p_0 +v_0 \tau\Delta t + 0.5a(\tau\Delta t)^2,
\end{equation}
with $\frac{\partial p_\s{max}}{\partial v_0} = \tau\Delta t$,
which is a function of $\tau$. Although $\tau$ also depends on $v_0$, it is piecewise constant and has zero derivative since it only takes integer values.
Comparing the gradients of $V^{h,\pi}_\infty$ with $V^{h,\pi}_{\infty,\Delta t}$ in \cref{fig:double_integrator_gradient}, we see a large error between the two with discontinuities in the discrete-time gradient in~$\Delta t$.
This is particularly problematic when the gradient is used in a gradient-based optimization algorithm such as \eqref{eq:cbf-qp}.

\noindent\textbf{Improved Time-Discretization using Cubic Splines. } To reduce the error in the time-discretized value function approximation \eqref{eq:tmp:b}, we propose to approximate $h(\vx_t^\pi)$ by fitting a cubic spline to the points $\{h(\vx_{k \Delta t}^\pi)\}_{k=0}^{H-1}$.
The $\max$ over the cubic spline can then be computed in closed-form by solving the roots of a quadratic to yield a better approximation of $\sup_{0 \leq t < T} h(\vx_t^\pi)$ than the naive maximization \eqref{eq:tmp:b}.
Intuitively, this resolves the gradient error due to integer-valued $\tau$ from the previous example because the maximum of the cubic spline can now happen \textit{between} timesteps.
\rev{We can also formally quantify the error in both the cubic spline value and its gradient.
Let $\tilde{h} : [0,\infty) \to \mathbb{R}$ denote the cubic spline approximation of $h$ as a function of time, and let $\tilde{V}^{h,\pi}_{\infty, \Delta t}(\vx_0) \coloneqq \sup_{t\geq0} \tilde{h}(t)$.
Using \cite[Chapter~5]{de1978practical}, we obtain
the error bounds
\begin{subequations} \label{eq:spline_bounds}
\begin{align}
    \norm{V^{h,\pi}_{\infty} - \tilde{V}^{h,\pi}_{\infty, \Delta t}} 
    &\leq \frac{1}{16} \Delta t^4 \max_{t \geq 0} \norm{ \frac{\diff^4}{\diff t } h(\vx_t) }, \\
    \norm{\nabla V^{h,\pi}_{\infty} - \nabla \tilde{V}^{h,\pi}_{\infty, \Delta t}} 
    &\leq \frac{1}{24} \Delta t^3 \max_{t \geq 0} \norm{ \frac{\diff^4}{\diff t } h(\vx_t) }.
\end{align}
\end{subequations}
In the previous example of the~\gls{di}, since $h$ is exactly quadratic, applying cubic splines results in zero gradient error (\cref{fig:double_integrator_gradient}).
If $\diff^4/\diff t^4\;  h(\vx_t)$ can be bounded, the bounds \eqref{eq:spline_bounds} can then be used to suitably modify~\eqref{eq:cbf-qp} to guarantee safety.} \hh{A larger $\Delta t$ will therefore result in a more conservative \gls{sf}.}

\begin{figure*}[htbp]
    \centering
    \vspace{1em}
    \includegraphics[width=0.75\linewidth]{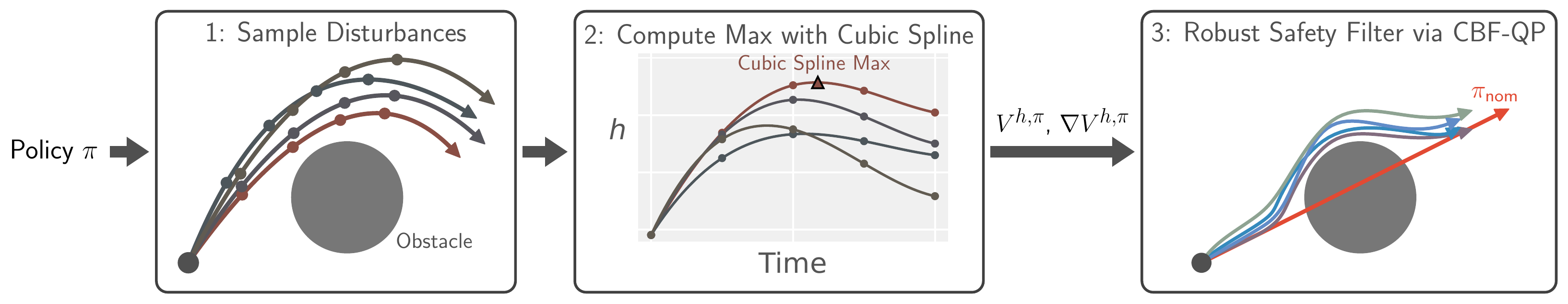}
    \small
    \caption{\small\textbf{Summary of RPCBF Algorithm. } Given a policy $\pi$, we sample disturbance trajectories, then compute the maximum $h$ with cubic splines to obtain $V^{h,\pi}$ and $\nabla V^{h,\pi}$ (using \hh{automatic differentiation}). This is used in \eqref{eq:cbf-qp} to obtain a robust \gls{sf}.}
    \label{fig:method}
    \vspace{-1.5em}
\end{figure*}
\subsection{Robust Extension of~\glspl{pcbf}}\label{sec:robust-ext}
We extend \gls{pcbf} to handle disturbances by defining the robust value function equivalent of \eqref{eq:value_fn} as
\begin{equation}\label{eq:rpcbf}
    \V^{h,\pi}_\infty(\vx_0) \coloneqq \sup_{t \geq 0} \sup_{\vd(\cdot)} h(\vx_t^\pi),
\end{equation}
it can be shown with similar proof \cite{so2024train,altarovici2013general} that $\V^{h,\pi}_\infty$ is a \textit{robust} CBF \cite{nguyen2021robust}, i.e., it satisfies \eqref{eq:cbf1} and, for $\cbf(\vx) \leq 0$,
\begin{equation}
    \sup_{\vd \in \cD} \inf_{\vu \in \cU} \nabla\cbf\T\big( f(\vx,\vd) + g(\vx,\vd)\vu \big) \leq -\alpha\big(\cbf(\vx)\big).\label{eq:cbf2_rob}
\end{equation}
\begin{algorithm}[htbp]
    \footnotesize
    \caption{\small Robust Policy CBF (RPCBF)}
    \label{alg:pncbf}
    \begin{algorithmic}[1]
        \State \text{\textbf{Input:}} Initial State $\vx_0$, Policy $\pi$, Constraint function $h$, Horizon $T=H\Delta t$ , Number of disturbance samples $N$
        \For {$i = 1:N$}
            \State Sample disturbance trajectory $\{ \vd^i_k \}_{k=1}^{H-1}$
            \State Rollout the policy $\pi$ on disturbed system \eqref{eq:system}
            \State Compute $\sup_{0\leq t<T} h(\vx_t^i)$ using cubic splines
        \EndFor
        \State Compute $\V^{h,\pi}_{T, N}(\vx_0)$ according to \eqref{eq:finiteHNV} 
        \State Compute the gradient $\nabla \V^{h,\pi}_{T, N}(\vx_0)$ using \hh{automatic differentiation}
    \end{algorithmic}
\end{algorithm}
Solving for robust controls that satisfy \eqref{eq:cbf2_rob} renders the zero sublevel set \textit{robust} forward-invariant \cite{jankovic2018robust}.
However, deriving the worst-case disturbance is generally intractable because it requires evaluating all possible disturbance trajectories.
Instead, we propose to only consider $N$ disturbance trajectories and take the worst-case out of the $N$ samples, resulting in the following~\gls*{rpcbf} approximation:
\begin{align}
    \V^{h,\pi}_T(\vx_0) &\approx \V^{h,\pi}_{T, N}(\vx_0) \coloneqq \max_{i=1,\cdots,N} \sup_{0 \leq t < T} h(\vx_t^i),\label{eq:finiteHNV}\\
    \dot{\vx}^i_t &= f(\vx^i_t, \vd^i_t) + g(\vx^i_t, \vd^i_{\rev{t}})\vu^i_t.\label{eq:xi}
\end{align}
We summarize our approach in \cref{alg:pncbf} and \cref{fig:method}. Note that \cref{alg:pncbf} must be executed once per control loop to obtain the value $V^{h,\pi}_{T,N}$ and gradient $\nabla V^{h,\pi}_{T,N}$ for the current state for use in \eqref{eq:cbf-qp}.
Different approaches can be implemented to perform informed sampling of disturbances.
For bounded disturbances, the worst-case scenario often occurs at the vertices of the disturbance set (e.g., for disturbance-affine dynamics).
Consequently, we choose to sample from a mixture of the uniform distribution $\mathcal{U}(\vd_\s{min}, \vd_\s{max})$ and the uniform distribution over the vertices.
Using better optimizers to approximate the worst-case samples will be left to future work.

While the finite-sample approximation does not guarantee robustness to any disturbance, it does ensure robustness to the specific sampled disturbances within the finite horizon. 
As the number of informed samples approaches infinity, the approximation increasingly captures the true worst-case scenarios.
\rev{Theoretical guarantees for this sampling-based approach could be established using random set theory~\cite{molchanov2005theory}, as demonstrated in~\cite{lew2021sampling}. }
\hh{Furthermore, statistical risk measures could be used instead of the worst-case out of $N$ samples. For instance, using \cite{akella2024sample} who bound the risk measure
evaluation of a random variable whose distribution is unknown.}
However, we leave this for future work.
For the simulation and hardware experiments below, we use \gls{pcbf} and \gls{rpcbf} to refer to their time-discretized finite-horizon and finite-sample approximations as described in this section.

\section{Simulation Experiments}\label{sec:sim_ex}
We evaluate the performance of (R)\gls{pcbf} in simulation for high relative degree systems with box control constraints.

\noindent\textbf{Baselines.} We compare against the following~\glspl{sf}, which similarly do not incorporate~\glspl{nn} in their approach.
\begin{itemize}
    \item \textbf{Handcrafted Candidate~\gls{cbf} (\acrshort{hocbf})~\cite{nguyen2016exponential, xiao2019control}:} We construct a \emph{candidate}~\gls{cbf} via a Higher-Order \gls{cbf} on~$h$ without considering input constraints.
    \item \textbf{Approximate Nominal \gls{mpc}-based \gls{psf} (MPC)~\cite{wabersich2021predictive}:} A trajectory optimization problem is solved, imposing the safety constraints while penalizing deviations from the nominal policy. We consider the undisturbed system without assuming access to a known robust forward-invariant set, thus not imposing a terminal constraint.
\end{itemize}

\noindent\textbf{Systems.} 
We consider \hhh{four} systems: a~\gls{di}, a Segway, \hhh{an F-16 fighter jet (ground collision avoidance problem) \cite{heidlauf2018verification, stevens2015aircraft},} and AutoRally \cite{ShieldMPPI}, \hhh{a} 1/5 autonomous vehicle.
On the~\gls{di}, we consider position bounds ($|p|\leq 1$), while the Segway asks for an upright handlebar and considers position bounds ($|\theta|\leq 0.3\pi$, $|p|\leq2)$, $\Delta t=0.1$. \hhh{For the F-16, safety is defined as box constraints on states like altitude. Since this system is not control-affine in the throttle, we leave the throttle as the output of a P controller, resulting in a $16$-dimensional state space and a $3$-dimensional control space.}
In AutoRally, a crash occurs when the car stops after hitting the track boundary, while a collision involves contact without stopping.
\hh{For each system, we define $J$ continuously differentiable constraint functions~$h_j$  tailored to the problem at hand. For example, for the \gls{di} system we set $h_0= p - 1$ and $h_1 = - (p+1)$. From these we derive $J$ corresponding \glspl{cbf}, which yield to $J$ constraints in \eqref{eq:cbf-qp}.}
For the~\gls{di} and the Segway, we assume unknown but bounded \rev{time-varying disturbances on the} mass, \hhh{for the F-16 unknown but bounded matched disturbances ($d=1$),} and for the AutoRally additive truncated Gaussian noise. 
\react{Keep in mind that our approach does not require designing/learning a new \gls{cbf} for different systems, disturbance assumptions, or input constraints, but simply requires swapping the dynamics, specifying the disturbance and constraints.}
During testing, we consider a constant zero-control nominal policy for the~\gls{di}, maximum acceleration for the Segway, \hhh{a PID controller for the F-16}, and~\gls{mppi} control~\cite{CCMPPI} for AutoRally.
\begin{figure}
    \centering
    \includegraphics[width=0.49\textwidth]{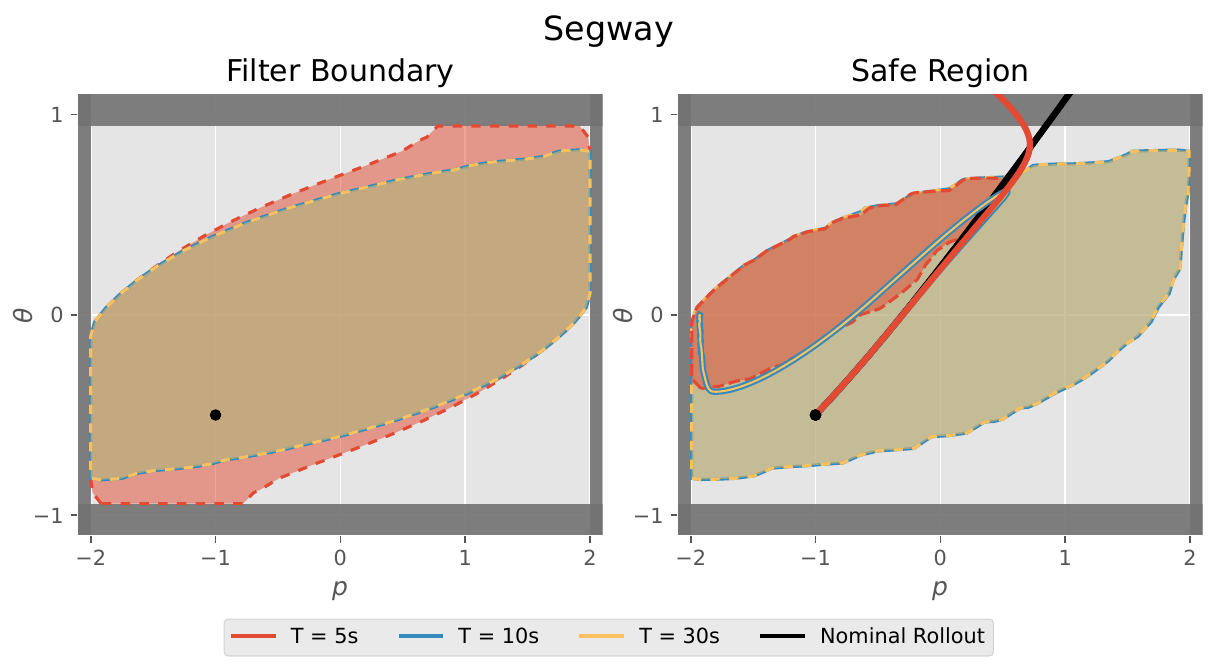}
    \caption{\small \rev{\textbf{Filter Boundary and Safe Region for \gls{pcbf}-\glspl{sf} with varying Horizon $\boldsymbol{T}$ on the Undisturbed Segway}}. We plot where the nominal policy affects the \gls{sf}'s output (Filter Boundary) and the states where the \gls{sf} ensures safety over $\bar{T} =30s$ (Safe Region). Trajectories from an initial state within the filter boundary (black dot, $\bullet$) are color-coded by horizon length of the \gls{pcbf}-\gls{sf}. A too-short horizon overapproximates the filter boundary, causing unsafe trajectories. } 
    \label{fig:segway_horizons}
    \vskip -5mm
\end{figure}
\begin{figure}[!tp]
    \centering
    \begin{subfigure}[b]{0.49\textwidth}
        \centering
        \includegraphics[width=\textwidth]{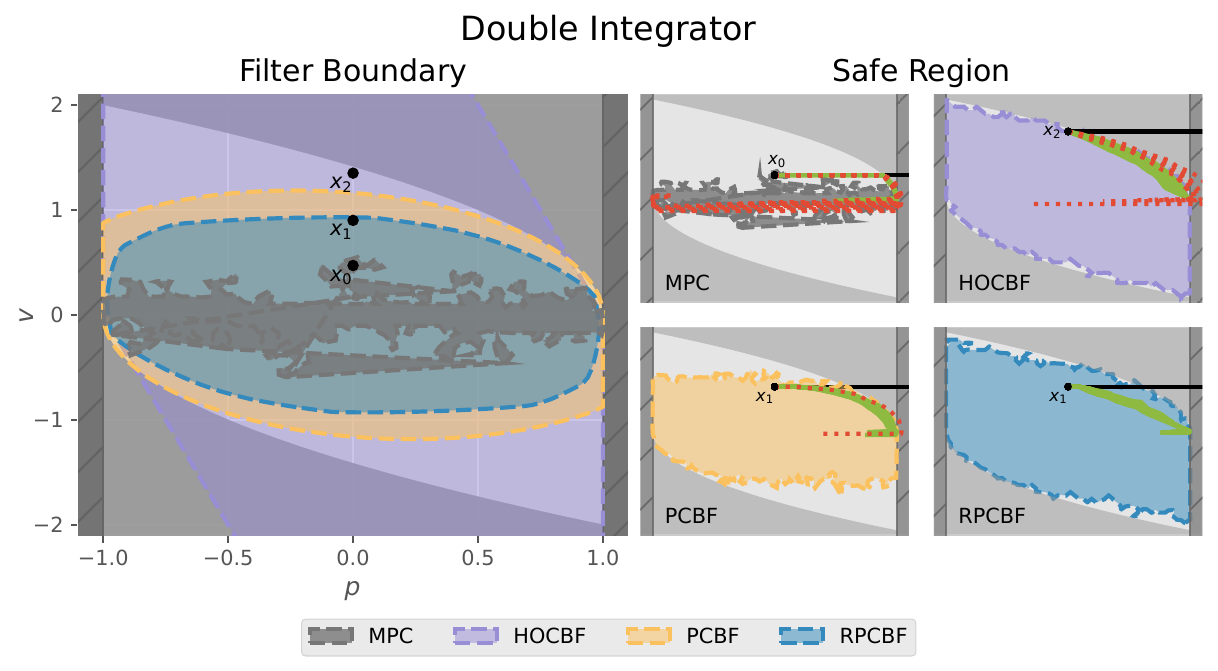}
        \caption{\small $T= 6.4s$, $N=100$ and $\bar{T}= 15s$.}
        \label{fig:subfig1}
    \end{subfigure}
    \hfill
    \begin{subfigure}[b]{0.49\textwidth}
        \centering
        \includegraphics[width=\textwidth]{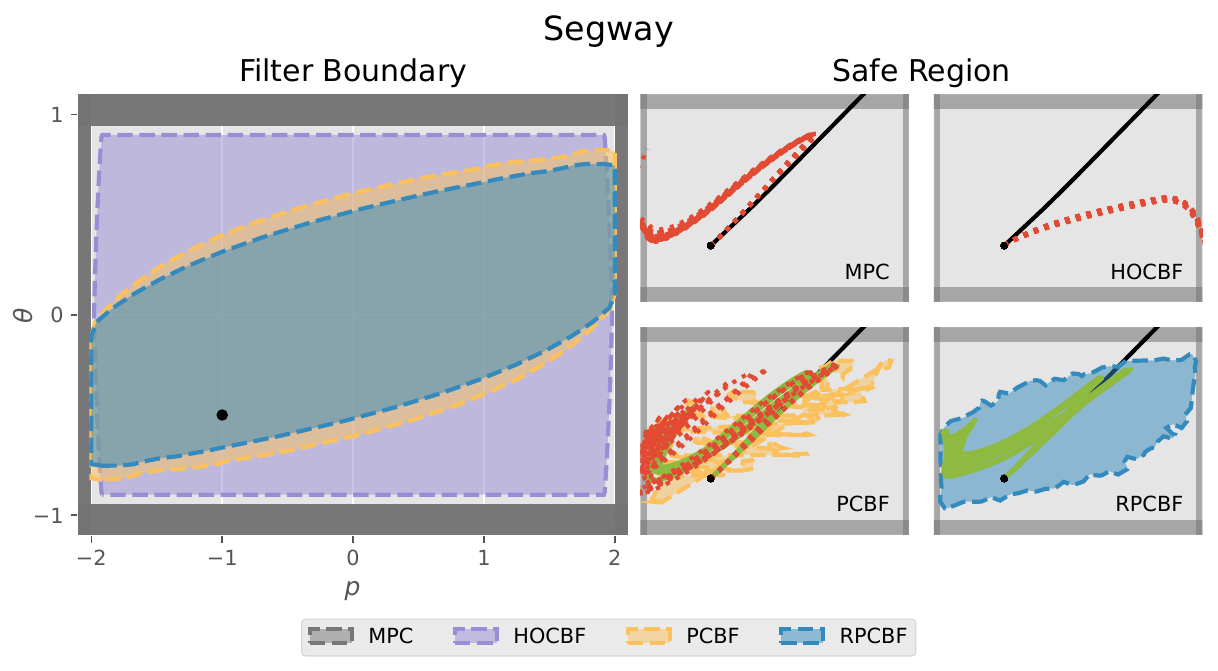}
        \caption{\small $T= 20s$, $N=100$ and $\bar{T}= 30s$}
        \label{fig:subfig2}
    \end{subfigure}
    \caption{\small \textbf{Comparison of Filter Boundary and Safe Region on DI (a) and Segway (b).} The true unsafe region for the undisturbed \gls{di} is shaded in \textcolor{gray}{gray}. (R)\glspl{pcbf} use horizon $T$ and $N$ samples to derive the value function. The safe region is determined for $\bar{T}$. Trajectories from selected initial states are shown for $\bar{N}=25$ sampled $\vd$ trajectories. \textcolor{ggRed}{Red dotted} and \textcolor{ggGreen}{green solid} lines indicate unsafe and safe trajectories, respectively, with the nominal trajectory in black.}
    \label{fig:dbint_comparison}
    \vskip -7mm
\end{figure}
\subsection{Influence of Horizon Length on Undisturbed Segway} \label{subsec:segway_horizon_length}
While the infinite-horizon policy value function is a \gls{cbf}, we use a finite-horizon approximation, making the \gls{sf} performance horizon-dependent. To illustrate this, we assess the impact of different horizon lengths on the \gls{pcbf}-\gls{sf}, see \cref{fig:segway_horizons}.
We plot the state space from where $\pi_\s{nom}$ can influence the output of the~\gls{sf} (\emph{Filter Boundary}) and from where the~\gls{sf} preserves safety (\emph{Safe Region}). For \gls{cbf}-based filters, the filter boundary is defined by the \gls{cbf}'s zero level set. The safe region is determined for a $\pi_\s{nom}$ by solving~\eqref{eq:cbf-qp} and rolling out the system over a horizon $\bar{T}=30s$. 
For a short~\gls{pcbf} horizon, i.e., $T=5s$, the true \gls{ci} set is overapproximated. Consequently, the~\gls{sf} fails to preserve safety, as illustrated by the resulting unsafe example trajectory. In contrast, a longer horizon of \rev{$T=10s$} provides a much closer approximation of the true \gls{ci} set. This is evident when comparing it to an even longer horizon, such as $T=30s$, which does not result in a visibly \rev{smaller} safe region, indicating that \rev{$T=10s$} is already sufficient.
\subsection{Behavior on Disturbed Double Integrator and Segway}
We explore the robustness of different~\glspl{sf}, examining the filter boundary and safe region, derived for one sampled disturbance trajectory per state, as shown in \cref{fig:dbint_comparison}.
We visualize rolled-out trajectories for~$\bar{N}$ sampled disturbance trajectories (uniformly sampled and on the vertices) from selected initial states within the filter boundary. 
On the~\gls{di}, only the \gls{rpcbf}-\gls{sf} achieves safety for all $\bar{N}$ sampled disturbance \rev{trajectories}. Since the \gls{rpcbf} accounts for the worst-case among the $N$ sampled disturbances, the filter boundary is more conservative.
On the Segway,~\gls{mpc} violates the safety constraints in all cases and hence has an empty filter boundary and safe region.
Only the \gls{rpcbf}-\gls{sf} achieves safe trajectories for all considered samples.
Next, we evaluate the 
(R)\gls{pcbf}-\glspl{sf} at uniformly distributed initial states,
see \cref{fig:dbint_safety}. The \gls{rpcbf}-\gls{sf} achieves safety for all evaluated states within its zero-level set and the sampled disturbances, \rev{while the \gls{pcbf} overapproximates the safe set.}
\begin{figure}
    \centering
    \includegraphics[width=0.49\textwidth]{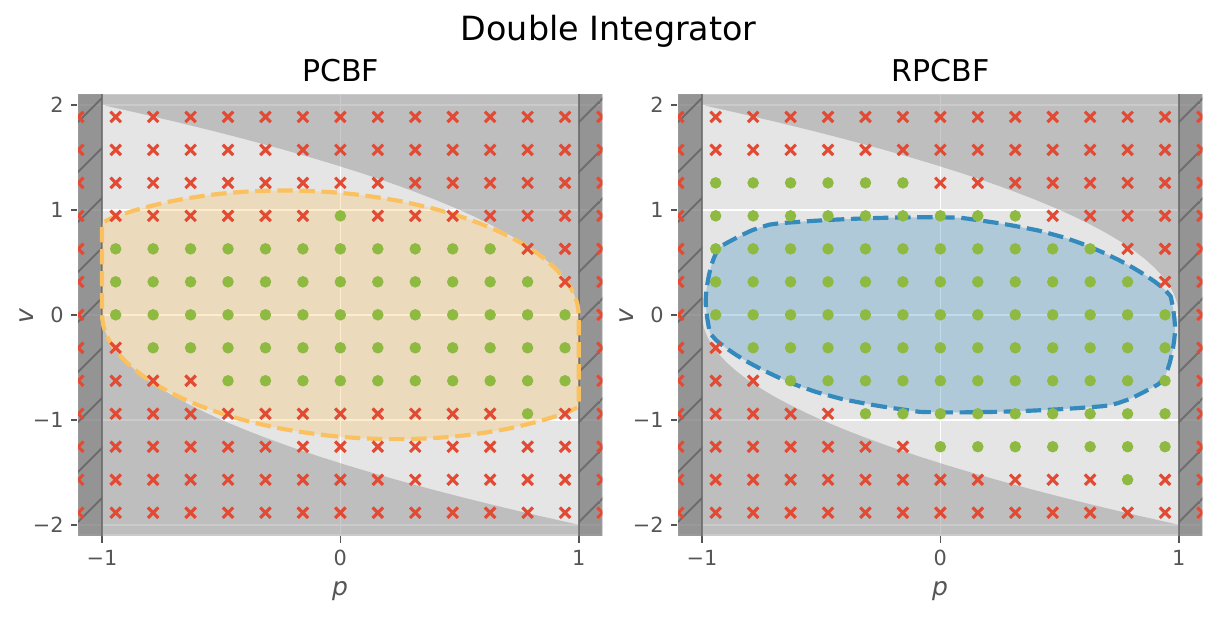}
    \caption{\small \textbf{Robust Safety Evaluation.} We plot the zero level set of the (R)\glspl{pcbf}. 
    \textcolor{ggGreen}{Green dots} indicate safe rollouts for all $\bar{N}=25$ sampled disturbance trajectories, while \textcolor{ggRed}{red crosses} indicate a failure in at least one trajectory. The \gls{rpcbf}-\gls{sf} achieved safety for all states within its zero level set for the sampled disturbance trajectories.}
    \label{fig:dbint_safety}
    \vskip -4mm
\end{figure}

\subsection{Simulations on AutoRally}
Finally, to assess the safety improvements brought about by the proposed (R)\gls{pcbf}, we integrate the \glsshort{hocbf} and the proposed methods with Shield-MPPI~(SMPPI)~\cite{ShieldMPPI, SCBFMPPI},
and test them on AutoRally. \Cref{fig:Autorally_traj} shows that SMPPI using~\gls{rpcbf} generates the safest trajectories.
The statistics of the safety performance of the controllers are shown in~\cref{Table: PerformanceComparison}.
\begin{figure}
    \centering
    \includegraphics[width=0.30\textwidth]{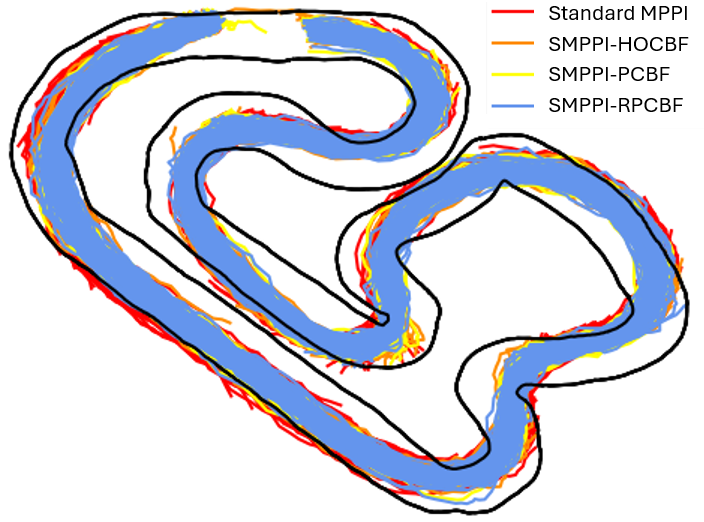}
    \caption{\small \textbf{Trajectory Comparisons on AutoRally.}
    SMPPI-RPCBF (in \textcolor{ARBlue}{blue}) leads to the tightest spread of states inside the track.
    }
    \label{fig:Autorally_traj}
    \vskip -2mm
\end{figure}
\begin{table}[!h]
\caption{\small \textbf{Collision \& Crash Rate on AutoRally.} \gls{mppi} causes the most collisions and crashes. SMPPI-HOCBF and SMPPI-\gls{pcbf} improve safety, while SMPPI-\gls{rpcbf} minimizes collisions \& crashes.}
\footnotesize
\centering
\begin{tabular}{cccc} \toprule
\textbf{Controller} & \textbf{Mean Collisions per Lap} & \textbf{Crash Rate} \\
\midrule
$\textrm{MPPI}$      & 4.53 & 0.80 \\
\midrule
$\textrm{SMPPI-HOCBF}$      & 1.38 & 0.15 \\
\midrule
$\textrm{SMPPI-PCBF}$      & 1.23 & 0.12 \\
\midrule
$\textrm{SMPPI-RPCBF}$      & \textbf{1.13} & \textbf{0.09} \\
\bottomrule
\end{tabular}
\label{Table: PerformanceComparison}
\vspace{-1em}
\end{table}
\subsection{\rev{Comparison of Computation Times}}\label{sec:compute}
\Cref{fig:comp_time} shows the real-time feasibility of~\gls{rpcbf}-\gls{sf} on the \gls{di} \hhh{and F-16}, highlighting how (component) computation times scale with increasing $T$ and $N$, evaluated on a laptop CPU.
\begin{figure}[t]
    \centering
    \includegraphics[width=0.5\textwidth]{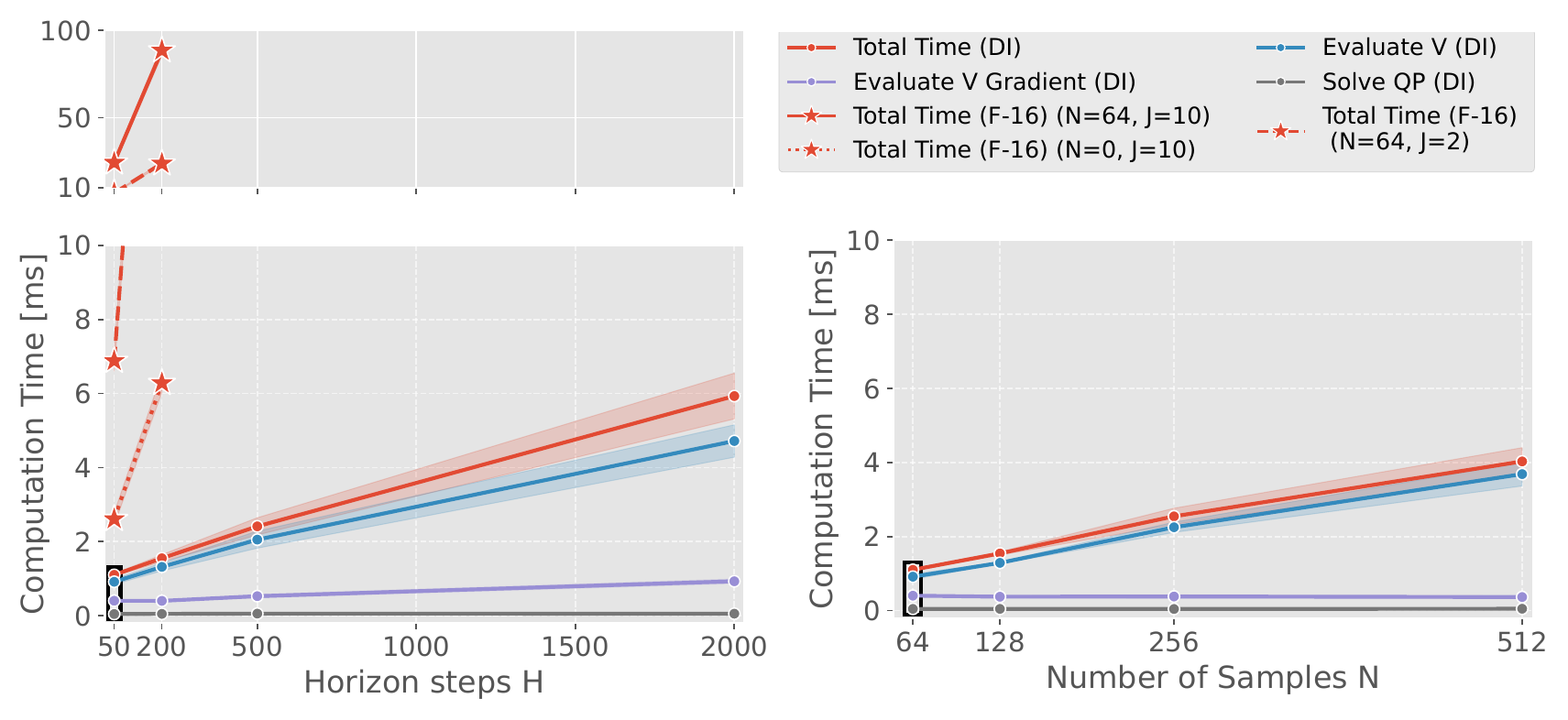}
    \caption{\small \hhh{\textbf{Per-Timestep Computation Times for \glslong{di} and F-16.}} Computation times for the total and individual components for varying \hh{$H$} ($N=64$) and $N$ (\hh{$H=50$}), including mean and standard deviation across initial conditions and timesteps. \hh{The black box highlights the settings considered in the hardware experiments.}}
    \label{fig:comp_time}
    \vskip -3mm
\end{figure}
\section{Hardware Experiments}
We conduct hardware experiments on the Crazyflie platform to test the robustness of the proposed~\gls{rpcbf}-\gls{sf} to real-world disturbances (see \cref{fig:teaser}).
\begin{figure}
    \centering
    \includegraphics[width=0.75\linewidth]{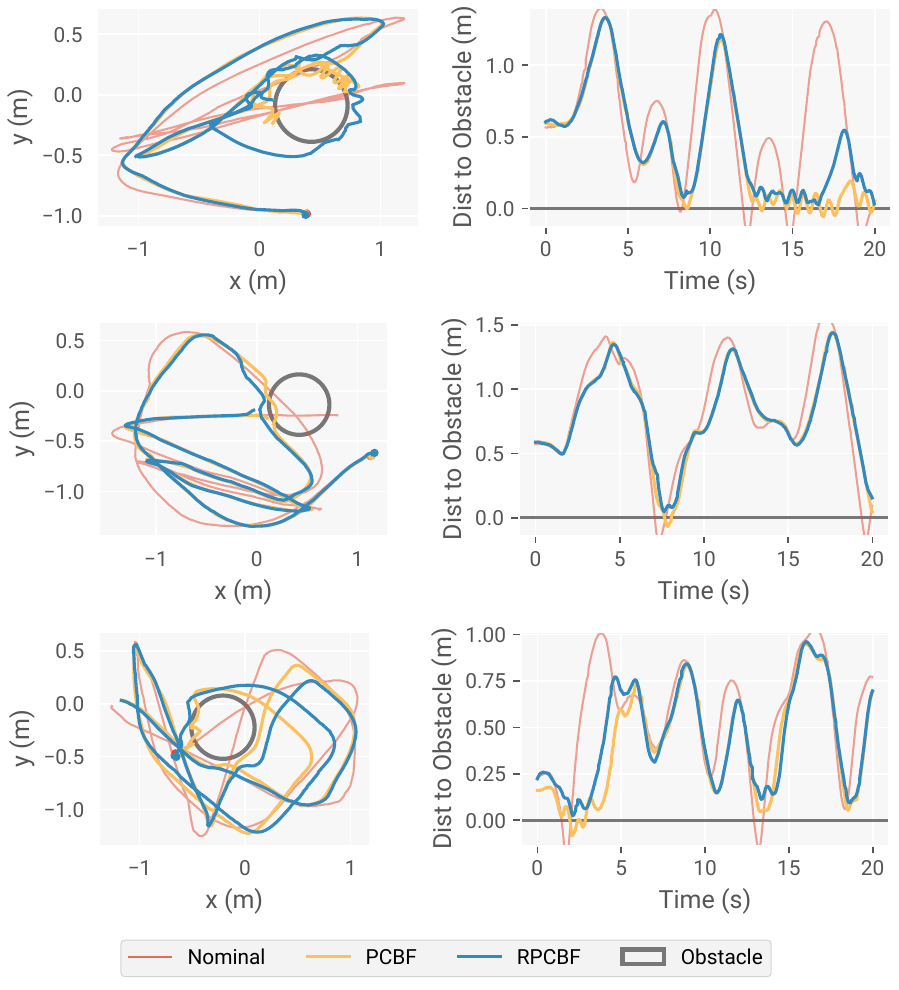}
    \caption{\small \textbf{Hardware Traces. } \gls{rpcbf}-\gls{sf} maintains safety despite model errors, while \gls{pcbf}-\gls{sf}, assuming perfect dynamics, collides.}
    \label{fig:cf_traj}
    \vskip -5mm
\end{figure}
Using the onboard position PID controller, we model the system as a~\gls{di}, assuming that position setpoints are converted to accelerations onboard and treat the model error as an acceleration disturbance.
We randomly generate nominal trajectories and treat the corresponding positions as the nominal control. A circular obstacle is placed at the densest part of the trajectory to encourage collisions. We use $T=5s$ (50 steps at $\Delta t = \qty{0.1}{\second}$) and $N=64$. The \gls{rpcbf}-\gls{sf} runs at a frequency of \qty{100}{\Hz} on a laptop. 

We run the (R)\gls{pcbf}-\glspl{sf} with $\alpha=5$ for $6$ nominal trajectories; \cref{fig:cf_traj} shows results for $3$ of them. The \gls{rpcbf}-\gls{sf} remains safe throughout, while the \gls{pcbf}-\gls{sf} always collides.

\begin{figure}
    \centering
    \vspace{1em}
    \includegraphics[width=0.55\linewidth]{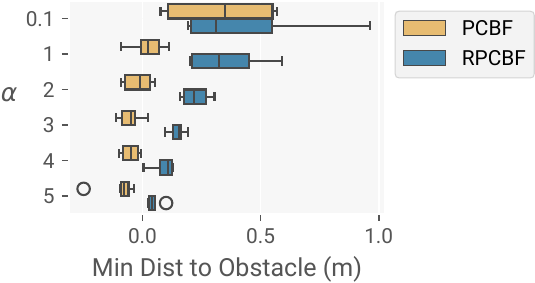}
    \caption{\small \textbf{Safety for different $\boldsymbol{\alpha}$. } Minimum distance to the obstacle over $6$ random nominal trajectories.
    While \gls{pcbf} is safe for small $\alpha$, it requires fine-tuning, whereas \gls{rpcbf} is safe for all tested $\alpha$.}
    \label{fig:cf_stats}
    \vskip -3mm
\end{figure}
We next vary the choice of the class-$\kappa$ function $\alpha$ and plot the results in \cref{fig:cf_stats}. While the non-robust~\gls{pcbf}-\gls{sf} does not collide with the obstacle when $\alpha$ is sufficiently small, this requires fine-tuning and is difficult to know beforehand. On the other hand, the \gls{rpcbf}-\gls{sf} is safe for all values of $\alpha$ we tested, allowing $\alpha$ to be used as a parameter that controls the behavior without also simultaneously affecting safety.

\section{Discussion and Conclusion}\label{sec:conclusion}
\rev{In this work, we proposed the~\acrfull{rpcbf}, a method for constructing robust \glspl{cbf} using the robust policy value function derived from rolling out a design policy. Subsequently, we introduced a real-time approximation that can be derived online, with conditions for its validity as a \gls{cbf}.} 
Simulation experiments demonstrate that a safety filter constructed using the~\gls{rpcbf} yields improved safety and more accurate estimation of the robust control-invariant set compared to existing methods. Hardware experiments on a quadcopter highlight the importance of accounting for model errors to ensure safety.

Future work will focus on analyzing the safety guarantees of the~\glspl{rpcbf} approximation, considering the finite-horizon, time-discretization, and sampling-based approach. 
\rev{Additionally, while the \gls{rpcbf} acts as a \gls{cbf} for any design policy if a long enough horizon is considered, conservativeness depends on the design policy. Deriving a policy to reduce conservativeness while maintaining infinite horizon guarantees is an important research direction.}
\hh{Moreover, extending our method to time-varying constraints and integrating real-time onboard perception are important directions for future work.}

\section*{Acknowledgement}
Any opinions, findings, conclusions or recommendations expressed in this material are those of the author(s) and do not necessarily reflect the views of the Under Secretary of Defense for Research and Engineering. 
© 2024 Massachusetts Institute of Technology.
Delivered to the U.S. Government with Unlimited Rights, as defined in DFARS Part 252.227-7013 or 7014 (Feb 2014). Notwithstanding any copyright notice, U.S. Government rights in this work are defined by DFARS 252.227-7013 or DFARS 252.227-7014 as detailed above. Use of this work other than as specifically authorized by the U.S. Government may violate any copyrights that exist in this work.  DISTRIBUTION STATEMENT A. Approved for public release. Distribution is unlimited.







\bibliographystyle{IEEEtran}
\bibliography{references}  

\end{document}